\documentclass[12pt]{amsart}
\usepackage{amsfonts,amssymb}

\newtheorem{theorem}{Theorem}

\newtheorem{lemma}{Lemma}

\newtheorem{proposition}{Proposition}

\theoremstyle{definition}
\newtheorem{definition}{Definition}
\theoremstyle{remark}

\newtheorem{remark}{Remark}

\newtheorem*{acknowledgements}{{\bf Acknowledgements}}

\newcommand{\bbc}{\mathbb{C}}
\newcommand{\cc}{\mathbb{C}}
\newcommand{\bs}{\backslash}
\newcommand{\re}{\mbox{Re}}
\newcommand{\diag}{\mbox{diag}}

\newcommand{\rr}{\mathbb{R}}
\newcommand{\adele}{\mathbb{A}}
\author{Omer Offen and Eitan Sayag}
\title[On representations distinguished by the symplectic group]
{On Unitary representations of $GL_{2n}$ distinguished by the symplectic group\date{}}
\begin{document}

\begin{abstract}
We provide a family of representations of $GL_{2n}$ over a p-adic field that admit a non-vanishing linear
functional invariant under the symplectic group (i.e. representations that are $Sp_{2n}$-distinguished).
While our result generalizes a result of M. Heumos and S. Rallis our methods, unlike their purely local technique,
relies on the theory of automorphic forms.
The results of this paper together with later works by the authors imply that the family of representations studied
in
this paper contains all irreducible, unitary representations of the general linear group that are distinguished by
the
symplectic group.
\end{abstract}
\maketitle

\section{Introduction}

Let $F$ be a $p$-adic field and let $G=GL_{2n}(F)$. We denote by
$$
    H_x=\{g \in G | {}^t g x g=x\}
$$ the symplectic group associated with the skew symmetric
matrix $x \in G$. We further denote by $H$ the group
$H_{\epsilon_{2n}}$ where
$$
    \epsilon_{2n}=\left(
        \begin{array}{cc}
            & w_n \\
            -w_n &
        \end{array}
    \right)
$$
and $w_n$ is the $n \times n$ permutation matrix with unit
anti-diagonal.
We are interested in smooth representations of the group $G$. For a
comprehensive theory see \cite{BZ}. We abuse notation and say that
such a representation is unitary when it is unitarizable.

\begin{definition}
A smooth representation $\pi$ of $G$ is called $H$-distinguished if
$$
    \mbox{Hom}_H(\pi,\cc)\not=0.
$$
\end{definition}
In this work we show that a certain family of irreducible, unitary
representations of $G$ are distinguished by the symplectic group
$H$. In upcoming works (\cite{os2}, \cite{os3}) we show, in particular, that
this family exhausts the irreducible, unitary, $H$-distinguished
representations of $G$.

Our interest in local symplectic periods is motivated by the work of
Klyachko over finite fields \cite{kly}. In \cite{MR91k:22036},
Heumos and Rallis began the study of an analogue for $p$-adic
fields. A survey of their work, with some motivation and the
relation to periods of automorphic forms, can also be found in
\cite{MR1216185}. Let us briefly describe the problem at hand.

Let $\psi$ be an additive character of $F$ and let $U_q$ denote
the group of upper triangular unipotent matrices in $GL_q(F)$. We
denote by
$$
    \psi_q(u)=\psi(u_{1,2}+\cdots+u_{q-1,q})
$$
the associated character of $U_q$. We will also denote by $H_{2q}$
the symplectic group $H_{\epsilon_{2q}}$. For $0 \le k \le [{q
\over 2}]$, let $M_{q,k}$ be the subgroup of $GL_q(F)$ of matrices
of the form
$$
    \left(
        \begin{array}{cc}
          u & X \\
          0 & h \\
        \end{array}
    \right)
$$
where $u \in U_{q-2k},\,h \in H_{2k}$ and $X \in M_{(q-2k) \times
2k}(F)$. We denote by $\psi_{q,k}$ the character
$$
    \psi_{q,k}
    \left(
        \begin{array}{cc}
          u & X \\
          0 & h \\
        \end{array}
    \right)=
    \psi_{q-2k}(u).
$$
We refer to the spaces
$$
    \mathcal{M}_{q,k}=\mbox{Ind}_{M_{q,k}}^{GL_q(F)}(\psi_{q,k})
$$
as \emph{Klyachko models}. The model $\mathcal{M}_{2n,n}$ is referred to as a symplectic model and the Klyachko
models
interpolate between a
Whittaker model and (if $q$ is even) a symplectic model. An
irreducible representation $\pi$ of $GL_q(F)$ is said to have the
Klyachko model $\mathcal{M}_{q,k}$ if
$$
    \mbox{Hom}_{GL_q(F)}(\pi,\mathcal{M}_{q,k}) \not=0.
$$
Note that a representation is $H$-distinguished if and only if it
has a symplectic model. In \cite{kly}, Klyachko showed that each
irreducible representation of $GL_q$ over a finite field has a
unique Klyachko model. In \cite{MR91k:22036}, Heumos-Rallis provide
evidence that every irreducible, unitary representation of $GL_q(F)$
has a Klyachko model. In fact they prove this fact for $q \le 4$.
They also show that irreducible, unitary representations can imbed
in at most one of the different Klyachko models $\mathcal{M}_{q,k}$.
We refer to \cite[p. 143]{MR1216185} for the local conjecture and
its global analogue. The present work is a step towards proving the
conjectures of [loc. cit.]. In \cite{os2}, we will show that any
irreducible unitary representation has a Klyachko model. Moreover,
we specify the model it has in terms of the Tadic parameter of the
representation. The exact description of a Klyachko model for any
unitary representation, together with the main result of the present
work and the disjointness of models \cite[Theorem
1]{os3}, imply that the representations that we consider
in Theorem \ref{symp} are precisely all irreducible, unitary
representations that are distinguished by $H$.

\begin{acknowledgements}
This work was done while both authors were Postdoctoral Fellows at the Weizmann Institute.
Both authors wish to thank Prof. Stephen Gelbart for his invitation and the Mathematics Department for excellent
working conditions. The first author was supported by the Sir Charles Clore Postdoctoral Fellowship at the Weizmann
Institute and the second author was supported by the Anne Stone Postdoctoral Fellowship at the Weizmann Institute.
\end{acknowledgements}

\section{Main Results}
To state our main theorem we briefly review Tadic's classification
of the unitary dual of $G$ \cite{MR88b:22021}. Denote by $\nu$ the
character $g \mapsto |\det g |$ on $GL_q(F)$ for any $q$. For
representations $\pi_i$ of $GL_{q_i}(F),\,i=1,\dots t$ and for
$q=q_1+\cdots+q_t$ we denote by $\pi_1 \times \cdots \times \pi_t$
the representation of $GL_q(F)$ obtained from $\pi_1 \otimes
\cdots \otimes \pi_t$ by normalized parabolic induction. For a
representation $\tau$ of $GL_q(F)$ and $\alpha \in \rr$ we denote
$\pi(\tau,\alpha)=\nu^\alpha \tau \times \nu^{-\alpha} \tau$. For
representations $\pi_i$ of $GL_{q_i}(F)$ set $\pi=\pi_1 \otimes
\cdots \otimes \pi_t$ and let $\lambda=(\lambda_1,\dots,\lambda_t)
\in \bbc^t$. We denote
$$
    \pi[\lambda]=\nu^{\lambda_1} \pi_1 \otimes \cdots \otimes
    \nu^{\lambda_t}\pi_t
$$
and
$$
    I(\pi,\lambda)=\nu^{\lambda_1} \pi_1 \times
    \cdots \times \nu^{\lambda_t} \pi_t.
$$
Let
$$
    \Lambda_{m}=({m-1
    \over 2},{m-3 \over 2},\dots,{1-m \over 2}) \in \rr^{m}.
$$
A representation of $GL_r(F)$ is called square integrable if its
matrix coefficients are square integrable modulo the center.
Square integrable representations are in particular unitary. For a
square integrable representation $\delta$ of $GL_r(F)$, the
representation $I(\delta^{\otimes m},\Lambda_{m})$ has a unique
irreducible quotient which we denote by $U(\delta,m)$. Let
$$
    B_u=
    \left\{
        U(\delta,m),\pi(U(\delta,m),\alpha): \delta\mbox{-square
        integrable,
        }m \in\mathbb{N}, |\alpha|< {1 \over 2}
    \right\}.
$$
A representation of the form $ \sigma_1 \times \cdots \times
\sigma_t $ where $\sigma_i \in B_u$, is irreducible and unitary.
Any irreducible, unitary representation of $GL_q(F)$ for some $q$
has this form and is uniquely determined by the multi-set of
$\sigma_i$'s up to reordering. This is the classification of
Tadic. Our main result is the following.
\begin{theorem}\label{symp}
Let $\pi=\sigma_1 \times \cdots \times \sigma_t \times \tau_{t+1}
\times \cdots \times \tau_s$ be a unitary representation of $G$,
such that $\sigma_i=U(\delta_i,2m_i) \in B_u$ and
$\tau_i=\pi(U(\delta_i,2m_i),\alpha_i) \in B_u$. Then $\pi$ is
$H$-distinguished.
\end{theorem}
In fact we prove in Proposition \ref{induction} that $\pi$ is
$H$-distinguished for a wider family of -- not necessarily unitary
-- representations. Theorem \ref{symp} is a generalization of a
result of Heumos-Rallis. They showed in \cite{MR91k:22036} that
the representations $U(\delta,2)$ are $H$-distinguished. Their
argument is the following. First, they construct a non-vanishing
$H$-invariant functional on $I(\delta \otimes \delta,\Lambda_2)$.
This representation has length 2 and its unique irreducible
subrepresentation has a Whittaker model. The existence of an
$H$-invariant functional on $U(\delta,2)$ is therefore a
consequence of the fact that irreducible generic representations
are not $H$-distinguished. This is a special case of \cite[Theorem
3.1]{MR91k:22036}. The method of proof of Heumos-Rallis does not
generalize directly to the case $m>1$. In Remark \ref{hrpf} we
explain where the difficulties lie. Our proof of Theorem
\ref{symp} is in two steps. We first use global methods to show
that the building blocks $U(\delta,2m)$ are $H$-distinguished. We
introduce in \eqref{iform} a non zero $H$-invariant functional
$j_H$ on $I(\delta^{\otimes 2m},\Lambda_{2m})$ and imbed
$I(\delta^{\otimes 2m},\Lambda_{2m})$ as the local component of a
certain global representation $I(\sigma^{\otimes
2m},\Lambda_{2m})$ induced from cuspidal. The functional $j_H$ is
then the corresponding local component of a certain factorizable,
period $\mathfrak{j}_H$ on $I(\sigma^{\otimes 2m},\Lambda_{2m})$
and $U(\delta,2m)$ is the local component of the unique
irreducible quotient $L(\sigma,2m)$ of $I(\sigma^{\otimes
2m},\Lambda_{2m})$. We then use the results of \cite{Odist} to
show that $\mathfrak{j}_H$ is not identically zero and factors
through $L(\sigma,2m)$. The second step consists of showing that
symplectic periods on the building blocks can be induced. Our
proof of this fact is rather technical. The idea, due to
Heumos-Rallis, is to apply Bernstein's principle of meromorphic
continuation. This requires convergence of a certain complicated
integral dependent on a complex parameter in some right
half-plane. We accomplish this in Lemma \ref{complicated} using an
integration formula of Jacquet-Rallis \cite{MR93b:22035}. In fact
we now know that the hereditary property of symplectic periods
follows from a recent work of Delorme-Blanc \cite{DelBla}.

\section{Symplectic period on the building blocks}\label{Hperiod}
Let $\delta$ be a square integrable representation of $GL_r(F)$
and let $n=mr$. We construct an explicit non-zero and
$H$-invariant linear form $l_H$ on the representation
$U(\delta,2m)$ of $G=GL_{2n}(F)$. For a permutation $w \in
\mathfrak{S}_{2m}$ in $2m$ variables, let $M(w)$ be the standard
intertwining operator
$$
    M(w):I(\delta^{\otimes 2m},\Lambda_{2m}) \to I(\delta^{\otimes
    2m},w\Lambda_{2m}).
$$
Let $w'=w'_{2m}$ be the permutation defined by
$$
        w'(2i-1)=i,\,w'(2i)=2m+1-i,\,i=1,\dots,m.
$$
Let $M$ be the standard Levi of type $(r,\dots,r)$ of $G$ and let
$M_H=M \cap H$. Up to a scalar, there is a unique $M_H$-invariant
form on $\delta^{\otimes 2m}$ which we denote by $l_{M_H}$.
Indeed, $l_{M_H}$ is given by the pairing of $\delta^{\otimes m}$
with its contragradiant. Let $K$ be the standard maximal compact subgroup of $G$ and set $K_H=K \cap H$.
The linear form
\begin{equation}\label{iform}
    j_H(\varphi)= \int_{K_H} l_{M_H} (M(w')\varphi (k))dk
\end{equation}
is a non-zero $H$-invariant form on $I(\delta^{\otimes 2m}
,\Lambda_{2m})$. Indeed, this is shown in \cite[\S 3]{Odist}, when
$\delta$ is any irreducible, generic, unitary representation. To
obtain a symplectic period on $U(\delta,2m)$ it is therefore
enough to show that the form $j_H$ factors through the unique
irreducible quotient $U(\delta,2m)$ of $I(\delta^{\otimes
2m},\Lambda_{2m})$.
\begin{remark}\label{hrpf}
If $\delta$ is supercuspidal, we can show that the representation
$I(\delta^{\otimes 2m},\Lambda_{2m})$ has a decomposition series
for which no factor (except $U(\delta,2m)$) is $H$-distinguished.
When $m=1$, the same is true for any square integrable $\delta$.
This was the key point in the proof of \cite[Theorem
11.1]{MR91k:22036}. For $m>1$ and $\delta$ square integrable this
is in general no longer true. The representation
$I(\delta^{\otimes 2m},\Lambda_{2m})$ may have more then one
decomposition factor which is $H$-distinguished. For this reason
the method of proof of Heumos-Rallis does not generalize directly
to $U(\delta,2m)$. To overcome this problem we use a global
approach.
\end{remark}

We imbed our local problem in a global setting. In order to
construct, locally, a non-vanishing symplectic period, we
construct a global, decomposable symplectic period and apply
\cite{Odist} to show that it factors through the unique
irreducible quotient.

\begin{proposition}\label{buildblock}
The form $j_H$ on $I(\delta^{\otimes 2m},\Lambda_{2m})$ factors
through $U(\delta,2m)$, i.e. it defines a non-zero $H$-invariant
form on $U(\delta,2m)$.
\end{proposition}
\begin{proof}
We start with the following lemma.
\begin{lemma}\label{global}
Let $\delta$ be a square integrable representation of $GL_r(F)$.
There is a number field $k$, a place $v$ of $k$ so that $F=k_v$
and a cuspidal automorphic representation $\sigma $ of
$GL_r(\adele_k)$ so that $\delta=\sigma_v$.
\end{lemma}
\begin{proof}[Proof of the lemma]
The lemma follows from the proof of Proposition 5.15 in
\cite{MR84j:12018}.
\end{proof}
Let $k,\,v$ and $\sigma$ be as in Lemma \ref{global}. Let $P$ be
the standard parabolic subgroup of $G$ with Levi $M$. For
$\lambda=(\lambda_1,\dots,\lambda_{2m}) \in \cc^{2m}$ denote
$$
    I(\sigma^{\otimes 2m},\lambda)=\mbox{Ind}_{P(\adele_k)}^{G(\adele_k)}(|\det |_{\adele_k}^{\lambda_1}\sigma
    \otimes
    |\det |_{\adele_k}^{\lambda_{2}}\sigma \otimes \cdots \otimes
    |\det |_{\adele_k}^{\lambda_{2m}}\sigma).
$$
Let $L(\sigma,2m)$ be the unique irreducible quotient of
$I(\sigma^{\otimes 2m},\Lambda_{2m})$. Then $I(\delta^{\otimes
2m},\Lambda_{2m})$ is the local component of $I(\sigma^{\otimes
2m},\Lambda_{2m})$ and $U(\delta,2m)$ is the local component of
$L(\sigma,2m)$ at $v$. Let
$$
        \mathfrak{j}_H(\varphi)=\int_{K_H} \int_{M_H \backslash M_H(\adele_k)^1}
        (M_{-1}(w')\varphi)
        (mk)\,dm\,dk
$$
where $M_{-1}(w')$ is the multi-residue at $\Lambda_{2m}$ of the
standard intertwining operator
$$
    M(w',\lambda):I(\sigma^{\otimes 2m},\lambda) \to
    I(\sigma^{\otimes 2m},w'\lambda).
$$
It is shown in \cite{Odist}, that the form $\mathfrak{j}_H$ is a
non-zero $H(\adele_{k,f})$-invariant form on $I(\sigma^{\otimes
2m},\Lambda_{2m})$, where $\adele_{k,f}$ is the ring of finite
ad\`eles of $k$. It is decomposable into local factors
$\mathfrak{j}_H=\otimes_w j_{H,w}$ and $j_{H,v}$ is proportional
to $j_H$ given by \eqref{iform}. Let $E_{-1}$ denote the
intertwining operator that projects $I(\sigma^{\otimes
2m},\Lambda_{2m}) \to L(\sigma,2m)$. It is also decomposable. In
\cite{OII} it is shown that $\mathfrak{j}_H$ factors through
$L(\sigma,2m)$, i.e. there is a linear form $\mathfrak{l}_H$ on
$L(\sigma,2m)$ that makes the following diagram commute
\begin{equation}\label{diagram}
    \begin{array}{ccc}
      I(\sigma^{\otimes 2m},\Lambda_{2m}) & \mathop{\to}\limits^{E_{-1}}  & L(\sigma,2m) \\
      \mathfrak{j}_H\searrow &  & \swarrow \mathfrak{l}_H \\
       & \cc &  \\
    \end{array}.
\end{equation}
Fix a decomposable element $\varphi_0=\otimes_w \varphi_{0,w} \in
I(\sigma^{\otimes 2m},\Lambda_{2m})$ such that
$\mathfrak{j}_H(\varphi_0)\not=0$. For each $\varphi_v \in
I(\delta^{\otimes 2m},\Lambda_{2m})$ denote $\varphi=\otimes_{w
\not=v }\varphi_{0,w} \otimes \varphi_v$. If $\varphi_v$ is in the
kernel of the projection $I(\delta^{\otimes 2m},\Lambda_{2m}) \to
U(\delta,2m)$ then $\varphi$ is in the kernel of $E_{-1}$ and
therefore $\mathfrak{j}_H(\varphi)=0=\mathfrak{j}_{H,v}(\varphi_v)
\prod_{w \not=v}\mathfrak{j}_{H,w}(\varphi_{0,w})$. Thus
$\mathfrak{j}_{H,v}(\varphi_v)=0$. This shows that
$\mathfrak{j}_{H,v}$ factors through $U(\delta,2m)$. The
proposition follows.
\end{proof}
We only needed to introduce global notation for the proof of
Proposition \ref{buildblock}. For the remainder of this work we
remain strictly in a local setting. Recall that $j_H$ is the
$H$-invariant form on $I(\delta^{\otimes 2m},\Lambda_{2m})$
defined by (\ref{iform}). It follows from Proposition
\ref{buildblock} that there exists an $H$-invariant form $l_H$ on
$U(\delta,2m)$ such that $j_H=l_H \circ M(w_{2n})$.
\section{Induction of the symplectic period}
In this section we fix irreducible, square integrable
representations $\delta_i$ of $GL_{r_i}(F),\,i=1,\dots,t$. We also
fix $\alpha_1,\dots,\alpha_t \in \rr$ and positive integers
$m_1,\dots,m_t$.
\begin{proposition}\label{induction}
The representation
\begin{equation}\label{even}
    J=\nu^{\alpha_1} U(\delta_1,2m_1) \times \cdots \times \nu^{\alpha_t} U(\delta_t,2m_t)
\end{equation}
is distinguished by $H$.
\end{proposition}
The rest of this work is devoted to the proof of Proposition
\ref{induction}. Let $k_i=m_i r_i$ and let ${\bf
k}=(2k_1,\dots,2k_t)$ be a partition of $2n$. Let $Q=LV$ be the
standard parabolic subgroup of $G$ of type ${\bf k}$, and let $x$
be the skew symmetric matrix
$$
    x=\diag(\epsilon_{2k_1},\dots,\epsilon_{2k_t}).
$$
We denote $Q_x=Q \cap H_x$ and let $P=MU$ be the standard
parabolic of $G$ of type
$$
    (\mathop{\overbrace{r_1,\dots,r_1}}\limits^{2m_1},\dots,\mathop{\overbrace{r_t,\dots,r_t}}\limits^{2m_t}).
$$
Its Levi component is $M=M_1 \times \cdots \times M_t$ where $M_i$
is the standard Levi of $GL_{2k_i}$ of type $(
r_i,\dots,r_i
)$. We denote by $M_{i,H}$ the intersection of $M_i$ with the
symplectic group $H_{{2k_i}}$ and by $K_{i,H}$ the intersection of
$H_{{2k_i}}$ with the standard maximal compact subgroup of
$GL_{2k_i}(F)$.  In \cite{Odist} we provided an $H$-filtration of
induced representations and a useful description of their
composition factors, using the geometric lemma of
Bernstein-Zelevinsky. The filtration of $J$ is parameterized by $Q
\bs G/H$. Let $l_i$ be the symplectic period on $U(\delta_i,2m_i)$
introduced in \S\ref{Hperiod}. It gives rise to a period on the
first composition factor coming from the open double coset. Let
$\eta \in G$ be such that $x=\eta \epsilon_{2n} {}^t\eta$. Then
$\eta H \eta^{-1}=H_x$ and $Q\eta  H$ is the open double coset. It
is a consequence of Frobenious reciprocity that on the subspace of
$J$, of functions supported on $Q\eta H$ we obtain a non-zero
$H$-invariant functional defined by the formula
$$
    l_H(\varphi)=\int_{(H \cap \eta^{-1}Q \eta) \bs H} (l_1 \otimes \cdots \otimes
    l_t)(\varphi(\eta h))\,dh=
$$
\begin{equation}\label{formula}
    \int_{Q_x \bs H_x} (l_1 \otimes \cdots \otimes
    l_t)(\varphi(h\eta))\,dh.
\end{equation}
However, this integral need not converge on the fully induced
space $J$. We follow the ideas of \cite{MR91k:22036} to bypass
this obstacle. We let
$$
    J_s=\mbox{Ind}_Q^G(\delta_Q^s \otimes(\nu^{\alpha_1} U(\delta_1,2m_1) \otimes \cdots \otimes
    \nu^{\alpha_t}
    U(\delta_t,2m_t)))
$$
where $\delta_Q$ is the modulus function of $Q$. Denote by
$l_{s,H}$ the linear form on $J_s$ defined by the right hand side
of (\ref{formula}). We show that for $\re s$ large enough and for
$ \varphi \in J_s$, the integral defining $l_{s,H}(\varphi)$ is
absolutely convergent. It will then follow from the uniqueness of
 symplectic periods \cite[Theorem 2.4.2]{MR91k:22036}, and from
Bernstein's principle of meromorphic continuation as used in
\cite[p. 277-278]{MR91k:22036}, that $J_s$ has a non-zero
symplecic period, which is a rational function of $q^s$, where $q$
is the cardinality of the residual field of $F$. This will provide
a non-zero symplectic period on $J=J_0$. Indeed, there will be an
integer $m$ so that $s^m l_{s,H}$ is holomorphic and non-zero at
$s=0$. We therefore only need to show that for $\re s
>>0$ and for $\varphi \in J_s$, the integral on the right hand side of (\ref{formula}) is
absolutely convergent. Let
$$
    I'_s=\mbox{Ind}_P^G(\delta_Q^s|_P
    \otimes (\nu^{\alpha_1}\delta_1^{\otimes 2m_1}[\Lambda_{2m_1}]
    \otimes \cdots \otimes \nu^{\alpha_t}\delta_t^{\otimes
    2m_t}[\Lambda_{2m_t}])).
$$
Let $E_i$ denote the projection from $I(\delta_i^{\otimes
2m_i},\Lambda_{2m_i})$ to $U(\delta_i,\Lambda_{2m_i})$. The
projection $E=E_1 \otimes \cdots \otimes E_t$ gives rise to a
projection $\tilde E_s:I'_s \to J_s$ given by
$$
    (\tilde E_s(f))(g)=E(f(g)).
$$
It is easy  to see that if $\varphi=\tilde E_s(f),\,f \in I'_s$
then
\begin{equation}\label{j'formula}
        l_{s,H}(\varphi)=\int_{Q_x \bs H_x} (j'_1 \otimes \cdots \otimes
    j'_t)(f(h\eta))\,dh
\end{equation}
where $j'_i=l_i \circ E_i$ is the non-zero symplectic period on
$I(\delta_i^{\otimes 2m_i},\Lambda_{2m_i})$ introduced in
(\ref{iform}).
We let $j'_{s,H}$ be the linear form on $I'_s$ defined by the
right hand side of (\ref{j'formula}). Let
$$
    I_s=\mbox{Ind}_P^G(\delta_Q^s|_P
    \otimes (\nu^{\alpha_1}\delta_1^{\otimes 2m_1}[w'_{2m_1}\Lambda_{2m_1}] \otimes \cdots \otimes
    \nu^{\alpha_t}\delta_t^{\otimes
    2m_t}[w'_{2m_t}\Lambda_{2m_t}]))
$$
and let $w'=\diag(w'_{2m_1},\dots,w'_{2m_t})$. Then $M(w')$ is the
standard intertwining operator from $I'_s$ to $I_s$. Making the
$j'_i$'s explicit we observe that
$$
    j'_{s,H} =j_{s,H}  \circ M(w'),
$$
where $j_{s,H}$ is the linear form on $I_s$ given by
\begin{equation}\label{jformula}
        j_{s,H}(\varphi)=
\end{equation}
$$
    \int_{Q_x \bs H_x} \int_{K_{1,H} \times
    \cdots \times K_{t,H}} (l_{M_{1,H}} \otimes \cdots \otimes l_{M_{t,H}})(f(\diag(k_1,\dots,k_t)h \eta))
    d(k_1,\dots,k_t) \,dh
$$
and $l_{M_{i,H}}$ is the $M_{i,H}$-invariant form on
$\delta_i^{\otimes 2m_i}$. It is left to prove the following.
\begin{lemma}\label{complicated}
For $\re s >>0$ and $f \in I_s$, the integral (\ref{jformula}) is
absolutely convergent.
\end{lemma}
\begin{proof}
It will be convenient to use the integration formula of
Jacquet-Rallis \cite{MR93b:22035} or rather its generalization to
$Q_x \bs H_x$ given in \cite{OII}. We will need to introduce some
new notation. We will try to minimize the notation and details and
focus only on the information we need for our proof of
convergence. More details regarding the integration formula can be
found in (\cite{OII}, \S5). For $Y=(y_1,\dots,y_m) \in F^m$ let
$$
    \|Y\|=\max_{i=1}^m(|y_i|)
$$
and let
$$
    \lambda(Y)=\max(\|Y\|,1).
$$
For $X=(X_1,\dots,X_{t-1})$ where $X_i \in
M_{2(k_{i+1}+\cdots+k_t) \times k_i}(F)$, we define a unipotent
matrix $\sigma_{\bf k}(X) \in G$ by recursion on $t$ as follows.
Let ${\bf k'}$ be the partition of $2n$ defined by ${\bf
k'}=(2k_1,\dots,2k_{t-2},2k_{t-1}+2k_t)$. Define
$$
    \sigma_{\bf{k}}(X)=
    \left(
        \begin{array}{cccc}
          1_{2(k_1+\cdots+k_{t-2})} &  &  &  \\
           & 1_{k_{t-1}} &  &  \\
           &  & 1_{k_{t-1}} &  \\
           &  & X_{t-1} & 1_{k_{t}}
        \end{array}
    \right)
    \sigma_{\bf k'}(X_1,\dots,X_{t-2}).
$$
For our purpose, it is enough to give the integration formula for
the $H_x$-invariant measure on $Q_x \bs H_x$, for functions $\phi$
on $G$ which are left $U$-invariant. There is a function
$\gamma(X)$ such that for functions $\phi$ as above we have
$$
    \int_{Q_x \bs H_x}\phi(h)dh=\int_{K_{H_x}} \int \gamma(X)
    \phi(\sigma_{\bf k}(X) k) dX dk
$$
where $K_{H_x}=K \cap H_x$. On the factor $\gamma(X)$ all we need
to know is that there are constant $c$ and $m$ such that
$$
    \gamma(X) \le c
    \left(
        \prod\limits_{i=1}^{t-1} \lambda(X_i)
    \right)^m.
$$
For $f \in I_s$ we therefore have
$$
    j_{s,H}(I_s(\eta^{-1})f)=\int_{K_{H_x}} \int \gamma(X) \int_{K_{1,H} \times
    \cdots \times K_{t,H}}
$$
$$
    (l_{M_{1,H}} \otimes \cdots \otimes l_{M_{t,H}})(f(\diag(k_1,\dots,k_t)\sigma_{\bf k}(X) k))
    d(k_1,\dots,k_t) dX dk.
$$
Since $f$ is $K$-finite, fixing a basis $\{f_j\}$ of $I_s(K) f$,
there are smooth functions $a_j$ on $K$ such that
$I_s(k)(f)=\sum_j a_j(k) f_j$. It follows that
$j_{s,H}(I_s(\eta^{-1})f)$ is the finite sum over $j$ of
$\int_{K_{H_x}}a_j(k) dk$ times
$$
    \int \gamma(X) \int_{K_{1,H} \times
    \cdots \times K_{t,H}} (l_{M_{1,H}} \otimes \cdots \otimes l_{M_{t,H}})(f_j(\diag(k_1,\dots,k_t)\sigma_{\bf
    k}(X)))
    d(k_1,\dots,k_t) dX.
$$
To prove the proposition it is therefore enough to show that the
integral
\begin{multline}\label{convstep}
    \int
    \left(
        \prod\limits_{i=1}^{t-1} \lambda(X_i)
    \right)^m \int_{K_{1,H} \times
    \cdots \times K_{t,H}}\\
    \left|
    (l_{M_{1,H}} \otimes \cdots \otimes l_{M_{t,H}})(f(\diag(k_1,\dots,k_t)\sigma_{\bf k}(X)))
    \right|
    d(k_1,\dots,k_t) dX
\end{multline}
is convergent. For any matrix $g$ we will denote by $\|\epsilon_i
g\|$ the maximum of the absolute values of the $i \times i$ minors
in the lower $i$ rows of $g$. For each $j \in [1,t],\,i \in
[1,2m_j]$, let $R_{i,j}=ir_j+2\sum_{q=j+1}^{t}k_q$. We write the
coordinates of each $\Lambda_{2m_j}$ as
$\Lambda_{2m_j}=(\mu_1^j,\dots,\mu_{2m_j}^j)$ (in fact the
convergence is proved for $\mu_i^j$ arbitrary). Let
$\mu=(\Lambda_{2m_1},\dots,\Lambda_{2m_t}) \in
\rr^{2(m_1+\cdots+m_t)}$. For $p \in P$ with diagonal blocks
$p_i^j \in GL_{r_j}(F),\,j \in [1,t],\,i \in [1,2m_j]$ we denote
$p^\mu=\prod_{i,j}|\det p_i^j|^{\mu_i^j}$. If we write an Iwasawa
decomposition of $g \in G$ with respect to $P$ as
$g=p(g)\kappa(g)$ then we have
$$
    f(g)=\delta_Q^s(p(g)) p(g)^{\mu+\rho_P}
    (\mathop{\otimes}\limits_{j=1}^t \delta_j^{\otimes
    2m_i})(p(g))f(\kappa(g))
$$
where $\rho_P$ is half the sum of positive roots with respect to
the parabolic $P$ of $G$. If $g=pk$ where $p \in P$ has diagonal
blocks denoted as before we may write
$$
    |\det p_i^j|={\|
    \epsilon_{R_{2m_j+1-i,j}}g\| \over \| \epsilon_{R}g\|}
$$
where $R=R_{2m_j-i,j}$ if $i<2m_j$ and $R=R_{1,j+1}$ otherwise. In
other words we may find $\lambda \in \rr^{2(m_1+\cdots+m_t)}$
dependent only on $\mu$ so that
$$
    f(g)=\delta_Q^s(p(g)) \prod_{i,j}\|
    \epsilon_{R_{i,j}}g\|^{\lambda_i^j}
    (\mathop{\otimes}\limits_{j=1}^t \delta_j^{\otimes
    2m_i})(p(g))f(\kappa(g)).
$$
The integral (\ref{convstep}) then becomes
\begin{multline}\label{forconv}
    \int
        \left(
            \prod_{i=1}^{t-1} \lambda(X_i)
        \right)^m
    \int_{K_{1,H} \times
    \cdots \times K_{t,H}}
    \delta_Q^s(p(\sigma_{\bf k}(X)))
    \prod_{i,j}\|
    \epsilon_{R_{i,j}}\diag(k_1,\dots,k_t)\sigma_{\bf
    k}(X)\|^{\lambda_i^j} \\
    \left|
    (l_{M_{1,H}} \otimes \cdots \otimes l_{M_{t,H}})((\otimes_{j=1}^t \delta_j^{\otimes
    2m_i})(p(\diag(k_1,\dots,k_t)\sigma_{\bf k}(X)))\right.
    \\
    \left. f(\kappa(\diag(k_1,\dots,k_t)\sigma_{\bf k}(X))))
    \right|
    d(k_1,\dots,k_t)\, dX.
\end{multline}
We first claim that the expression in the absolute value is now
bounded, independently of $k_1,\dots,k_t$ and $X$. Indeed, $f$
being smooth, obtains only finitely many values on $K$ and
therefore it is enough to bound
$$
    (l_{M_{1,H}} \otimes \cdots \otimes l_{M_{t,H}})((\otimes_{j=1}^t \delta_j^{\otimes
    2m_i})(p(\diag(k_1,\dots,k_t)\sigma_{\bf k}(X))) v
$$
for any $v$ in the space of $\otimes_{j=1}^t \delta_j^{\otimes
2m_j}$. We may further assume that $v$ decomposes as $v=v_{1,1}
\otimes v_{1,2} \otimes \cdots \otimes v_{t,1} \otimes v_{t,2}$
where $v_{i,j}$ lies in the space of $ \delta_i^{\otimes m_i}$ for
$j=1,2$. For $p \in M$ we denote by $p_i^j \in GL_{r_i}(F)$ its
diagonal blocks as before. The map
$$
    p \mapsto
    (l_{M_{1,H}} \otimes \cdots \otimes l_{M_{t,H}})((\mathop{\otimes}\limits_{j=1}^t \delta_j^{\otimes
    2m_i})(p)v)
$$
is a matrix coefficient of the unitary representation
$\mathop{\otimes}\limits_{j=1}^t \delta_j^{\otimes m_j}$ evaluated
at
$$
    \diag((\tilde p_{2m_1}^1)^{-1}p_1^1,\dots,(\tilde p_{m_1+1}^1)^{-1}p_{m_1}^1,
    (\tilde p_{2m_2}^2)^{-1}p_1^2,\dots,(\tilde p_{m_t+1}^t)^{-1}p_{m_t}^t).
$$
Here $\tilde g=w_q {}^t g^{-1} w_q$ for $g \in GL_q(F)$. Matrix
coefficients of unitary representations are bounded. It is
therefore enough to show that for $\re s $ large enough the
expression
\begin{multline}\label{converge}
    \int
        \left(
            \prod_{i=1}^{t-1} \lambda(X_i)
        \right)^m
    \int_{K_{1,H} \times
    \cdots \times K_{t,H}}
    \delta_Q^s(p(\sigma_{\bf k}(X)))\\
    \prod_{i,j}\|
    \epsilon_{R_{i,j}}\diag(k_1,\dots,k_t)\sigma_{\bf
    k}(X)\|^{\lambda_i^j} \times
    d(k_1,\dots,k_t) dX
\end{multline}
converges. In order to bound the integrand in \eqref{converge}, we
will use the following two claims.

{\bf Claim 1:} There exist an $N$ such that
$$
    1 \le \|\epsilon_{R_{i,j}}\diag(k_1,\dots,k_t)\sigma_{\bf k}(X)\|
    \le         \left(
            \prod\limits_{i=1}^{t-1} \lambda(X_i)
        \right)^N
$$
and

{\bf Claim 2:}
$$
    \delta_Q^s(p(\sigma_{\bf k}(X))) \le \left(
            \prod\limits_{i=1}^{t-1} \lambda(X_i)
        \right)^{-s}.
$$
The upper bound in {\bf Claim 1} is obvious. We show the lower
bound. To avoid ambiguity of notation let us denote by $k_{i,H}$
the elements of $K_{i,H}$. Note that the lower $R_{1,j}$ rows of
$\diag(k_{1,H},\dots,k_{t,H})\sigma_{\bf k}(X)$ have the form
$$
    \left(
    \begin{array}{ccc}
     * & k_{j,H} & 0 \\
      {}* & * & \diag(k_{j+1,H},\dots,k_{t,H})\sigma_{(2k_{j+1},\dots,2k_t)}(X_{j+1},\dots,X_{t-1})
    \end{array}
    \right)
$$
where we put $*$ in each block that will play no role for us. For
each $i \in [1,2m_j]$ there is a $ir_j \times ir_j$ minor $A$ in
the lower $ir_j$ rows of $k_{j,H}$ of absolute value $1$. Together
with the lower right $2(k_{j+1}+\dots k_t) \times 2(k_{j+1}+\dots
k_t)$-block of $\diag(k_{1,H},\dots,k_{t,H})\sigma_{\bf k}(X)$ we
get that
$$
    \left(
        \begin{array}{cc}
            A & 0 \\
            {}* & \diag(k_{j+1,H},\dots,k_{t,H})\sigma_{(2k_{j+1},\dots,2k_t)}(X_{j+1},\dots,X_{t-1})
        \end{array}
    \right)
$$
is a $R_{i,j} \times R_{i,j} $ minor of absolute value $1$ in the
lower $R_{i,j}$ rows of the matrix
$\diag(k_{1,H},\dots,k_{t,H})\sigma_{\bf k}(X)$. This shows {\bf
Claim 1}. To show {\bf Claim 2}, we note that if $|\det g|=1$ then
$$
    \delta_Q(p(g))=\prod_{j=1}^{t-1} \|\epsilon_{R_{1,j}}
    g\|^{-2k_{t+1-j}-2k_{t-j}}.
$$
It can be proved as in \cite[Lemma 5.5]{OII}, that
$$
    \|\epsilon_{R_{1,j}}p(\sigma_{\bf k}(X))\| \ge \lambda(X_{t-j}).
$$
{\bf Claim 2} readily follows. Using the two claims, we bound the
integral (\ref{converge}) by replacing each term $\|
    \epsilon_{R_{i,j}}\diag(k_1,\dots,k_t)\sigma_{\bf
    k}(X)\|^{\lambda_i^j}$ by $1$ if $\lambda_i^j \le 0$ and by a certain fixed and large enough power of
$\left(
       \prod\limits_{i=1}^{t-1} \lambda(X_i)
\right) $ otherwise. It is shown in \cite{MR93b:22035} that for
any $q$ the integral
$$
    \int_{F^q} \lambda(Y)^{-s} dY
$$
is convergent for $s
>>0$. The lemma therefore follows from the two claims.
\end{proof}
This concludes the proof of Proposition \ref{induction} and in
particular also of Theorem \ref{symp}.

\end{document}